\documentclass[12pt, reqno]{amsart}
\usepackage{latexsym}
\usepackage{amssymb} 
\usepackage{mathrsfs}
\usepackage{amsmath}
\usepackage{latexsym}
\usepackage{delarray}
\usepackage{amssymb,amsmath,amsfonts,amsthm,mathrsfs}
\usepackage{blkarray}

\usepackage{hyperref}
\renewcommand\eqref[1]{(\ref{#1})} 

 \newtheorem{thm}{Theorem}[section]
 \newtheorem{cor}[thm]{Corollary}
 \newtheorem{lem}[thm]{Lemma}
 \newtheorem{prop}[thm]{Proposition}
 \theoremstyle{definition}
 \newtheorem{defn}[thm]{Definition}
 \theoremstyle{remark}
 \newtheorem{rem}[thm]{Remark}
 \newtheorem*{ex}{Example}
 \numberwithin{equation}{section}

\newcommand{\half}{\frac{1}{2}}

\newcommand{\sdual}{\frac{1}{p}+\frac{1}{q}=1}
\newcommand{\ene}{\mathbb{N}}
\newcommand{\bba}{\mathcal{B}}

\newcommand{\er}{\mathbb{R}}
\newcommand{\ce}{\mathbb{C}}
\newcommand{\zet}{\mathbb{Z}}
\newcommand{\zn}{\mathbb{Z}^n}
\newcommand{\tn}{\mathbb{T}^n}
\newcommand{\T}{\mathbb{T}^1}
\newcommand{\To}{\mathbb{T}}

\newcommand{\fou}{\mathcal{F}}

\newcommand{\bi}{\begin{itemize}}

\newcommand{\ei}{\end{itemize}}
\newcommand{\be}{\begin{enumerate}}
\newcommand{\ee}{\end{enumerate}}
\newcommand{\beq}{\begin{equation}}
\newcommand{\eq}{\end{equation}}

\newcommand{\Dcal}{\mathcal{D}}
\newcommand{\Dp}{\mathcal{D}_p}
\def\p#1{{\left({#1}\right)}}

\def\Hcal{{\mathcal H}}

\DeclareMathOperator{\Tr}{Tr}
\DeclareMathOperator{\Det}{Det}

\def\Rn{{{\mathbb R}^n}}
\def\Tn{{{\mathbb T}^n}}
\def\Zn{{{\mathbb Z}^n}}
\def\T{{{\mathbb T}^1}}
\def\N{{{\mathbb N}}}

\def\SU2{{{\rm SU(2)}}}
\def\lapsu2{{{\mathcal L}_\SU2}}

\begin{document}
\setcounter{page}{1}

\title[Determinants and Plemelj-Smithies formulas]
 {Determinants and Plemelj-Smithies formulas}

\author[Duv\'an Cardona]{Duv\'an Cardona}

\address{%
	Ghent University\\
	Department of Mathematics: Analysis, Logic and Discrete Mathematics\\
	Krijgslaan 281, Building S8 \\
	B 9000 Ghent\\
	Belgium}

\email{duvan.cardonasanchez@ugent.be}

\author[Julio Delgado]{Julio Delgado}

\address{%
	Universidad del Valle
	Departamento de Matem\'aticas\\
	Calle 13 100-00\\
	Cali-Colombia}
\email{delgado.julio@correounivalle.edu.co}

\author[Michael Ruzhansky]{Michael Ruzhansky}

\address{%
	Ghent University\\
	Department of Mathematics: Analysis, Logic and Discrete Mathematics\\
	Krijgslaan 281, Building S8 \\
	B 9000 Ghent\\
	Belgium}

\email{Michael.Ruzhansky@ugent.be}

\address{%
	Queen Mary University of London\\
	School of Mathematical Sciences\\
	Mile End Road\\
	London E1 4NS\\
	United Kingdom}
\email{m.ruzhansky@qmul.ac.uk}

 \allowdisplaybreaks



\date{\today}
\thanks{The authors are supported  by the FWO  Odysseus  1  grant  G.0H94.18N:  Analysis  and  Partial Differential Equations. The second author was also supported by the Grant CI71234 Vic. Inv. Universidad del Valle and by Leverhulme Grant RPG-2017-151. The third author is also supported by the Leverhulme Grant RPG-2017-151 and by EPSRC grant 
EP/R003025/1.}

\begin{abstract}
We establish Plemelj-Smithies formulas for determinants  in different algebras of operators. In particular we define a Poincar\'e type determinant for operators on the torus $\Tn$ and deduce formulas
 for determinants of periodic pseudo-differential operators in terms of the symbol. On the other hand, 
by applying a recently introduced notion of invariant operators  relative to
 fixed decompositions of Hilbert spaces we also obtain formulae for determinants with respect to the trace class. The analysis  makes use of the corresponding notion of full matrix-symbol. We also derive explicit formulas for determinants associated to elliptic operators on compact manifolds, compact Lie groups, and on homogeneous vector bundles over compact homogeneous manifolds. 
\end{abstract}

\maketitle
\tableofcontents
\allowdisplaybreaks
\section{Introduction}
In this work we study determinants of operators in different algebras and  obtain the corresponding  Plemelj-Smithies formulas. A special case that we will consider is a   Poincar\'e type determinant for operators on the torus $\Tn$. To do so we apply the Fourier transform on $\Tn$ leading to an isomorphism between operators on  $\Tn$ and operators on $\Zn$. By considering pseudo-differential operators on the torus we also obtain formulae in terms of symbols. Regarding compact manifolds we apply a recently introduced notion of invariance (cf. \cite{dr14a:fsymbsch}). This notion  is relative to a fixed decomposition of a complex separable Hilbert space. In particular one can consider such decompositions  into eigenspaces associated to an elliptic operator on a compact manifold or on a vector bundle over a compact manifold. The classes of operators covered in this work also include $G$-invariant operators on homogeneous vector bundles allowing applications of the invariant notion in \cite{dr14a:fsymbsch} to Dirac type operators on the bundle of forms, on some spin structures and on some twisted geometric structures (cf. \cite{Bott65}).\\

Some of the Plemelj-Smithies formulas for determinants that we are going to state here can be seen as complementary to the ones stated for traces in  \cite{dr:suffkernel}, \cite{dr14a:fsymbsch}, \cite{dr13:schatten} and \cite{dr13a:nuclp}. The determinants of matrices of infinite order were used for the first time by G. W. Hill in his study of the motion of  the lunar perigee (cf. \cite{hill1:cl}). 
 A first rigorous approach to such determinants was introduced by Poincar\'e in his work on the Hill's equation  \cite{poi:hill}.  A general point of view to define the determinant of $I+A$ consists in considering $A$ as an operator in a class endowed with a trace.  In particular the Plemelj-Smithies formulas express the determinant in terms of traces. Here we are going to consider the point of view of embedded algebras introduced by I. Gohberg, S. Goldberg and N. Krupnik (cf. \cite{goh:trace}). Here we will work within that  framework and specifically the important examples of trace class operators on a Hilbert space, the algebra of matrices in $\ell^1(\zet\times\zet)$ viewing them as operators on $\ell^1(\zet)$ and domain for the Poincar\'e determinant, and the nuclear operators on Banach spaces.  \\
 
The applications of Fredholm determinants in the  analysis of differential equations has actively motivated its investigation in the last decades 
 see e.g. \cite{Zhao-Barnett,Bothner-Its:CMP-2014, Borodin-Corwin-Remenik,Gesztesy-Latushkin-Zumbrun,McKean:CPAM-2003}.
 Furthermore, a systematic study of numerical computations for Fredholm determinants  was initiated by Bornemann \cite{Bornemann:MC-2010}.
  We refer to these papers for further
references and motivations. \\

The structure of this work is as follows. In Section \ref{sec2} we present the notion of a determinant on Banach spaces, the notion of $r$-nuclearity of Grothendieck and the calculus of invariant operators developed in \cite{dr14a:fsymbsch}. Finally, in Section \ref{sec3} we establish our main results, first, for pseudo-differential operators on the torus, invariant operators on Hilbert spaces, trace class operators in $L^2(M)$ (for $M$ a closed manifold), and at the end of the section for trace class $G$-invariant operators on homogeneous vector-bundles.

 
\section{Traces and Determinants}\label{sec2}
In this section we define the concepts of  traces and determinant in the setting of  embedded algebras according to \cite{goh:trace} and we refer the reader to it for a more detailed discussion. \\

Let  $\mathcal{B}$ be a Banach space, we denote by  $\mathcal{F}(\bba)$ the space of bounded finite rank operators on  $\bba$.  
  We also denote by $\mathcal{L}(\bba)$ the $C^*$-algebra of bounded linear operators on $\bba$.  
 We now briefly recall the definition of trace and determinants for finite rank operators. The following fundamental properties of trace and determinants for finite rank operators are mainly consequences of the ones for the finite square matrix setting. 
  
\begin{lem} Let $\bba$ be a Banach space and $F\in \mathcal{F}(\bba)$. Then  $\bba$ can be decomposed as a direct sum
	\beq B=M_F\oplus N_F ,\label{dmnfg}\eq
	where $F(M_F)\subset M_F$ and $N_F\subset Ker F.$
\end{lem}

The decomposition \eqref{dmnfg} allows us to write the operators $F$ and $I+F$ as $2\times 2$ matrices of the form
\beq F=\begin{bmatrix} 
	F_1 & 0 \\
	0 & 0 
\end{bmatrix},\,\, I+F=\begin{bmatrix} 
I+F_1 & 0 \\
0 & I_2 
\end{bmatrix}.
\label{mm35}\eq
Here $F_1=F|_{M_F}$ is the restriction of $F$ to the finite dimensional subspace $M_F$ and $I, I_1, I_2$ denote the identity operators in $\bba, M_F, N_F$ respectively. Since  $M_F$ is a finite dimensional space, the functionals $\Tr{F_1}$ and $\Det(I+F_1)$ are well defined. Using this we can define:
\beq \Tr(F):=\Tr(F_1), \,\, \Det (I+F):=\Det(I+F_1).\label{ddtrf5}
\eq
These definitions are independent of the choice of the subspace $M_F$. Indeed, this is due to the following important formulas in terms of eigenvalues:
\beq \Tr(F_1)= \sum\limits_{j=1}^m \lambda_j(F_1),\,\, \Det(I+F_1)=\prod_{j=1}^{m}(1+\lambda_j(F_1)) ,  
\eq
where $m=\dim M_F$ and $\lambda_1(F_1),\dots,\lambda_m(F_1)$ are the eigenvalues of $F_1$ counted according to their algebraic multiplicites. Moreover, an application of \eqref{mm35} and the Jordan decomposition of $F_1$ shows that the nonzero eigenvalues of $F$ do not depend on the choice of $M_F$ and  
\beq \Tr(F)= \sum\limits_{j=1}^m \lambda_j(F),\,\, \Det(I+F)=\prod_{j=1}^{m}(1+\lambda_j(F)).  
\eq
Therefore the Definition \eqref{ddtrf5} is  independent of the choice of the subspace $M_F$. Moreover, we see that $I+F$ is invertible in $\mathcal{L}(\bba)$ if and only if $\Det(I+F)\neq 0$.

The functionals $\Tr(\cdot)$ and $\Det(\cdot)$ enjoy some fundamental properties:
\begin{itemize}
    \item The trace  $\Tr(\cdot)$ is linear on $\mathcal{F}(\bba)$ and 
\beq\Tr(AB)=\Tr(BA).\label{comtr}\eq
\item The determinant  has the following multiplicative property
\beq \Det((I+A)(I+B))=\Det(I+A)\Det(I+B)
\eq
and 
\beq  \Det(I+AB)=\Det(I+BA).
\eq
\end{itemize}

In order to extend the trace and determinants to larger classes of operators, we recall the notion of algebras that we are going to use. We note that the functionals trace and determinant are not continuous with respect to the operator norm. This fact motivates 
 the introduction of the following concept of algebra. 

\begin{defn} A subalgebra $\mathcal{D}$ of $\mathcal{L}(\bba)$ is {\em  continuously embedded } in $\mathcal{L}(\bba)$ if the following two conditions hold:\\
	
	(i) There exists a norm $\|\cdot\|_{\mathcal{D}}$ on $\mathcal{D}$ and a constant $C>0$ such that
	\beq\|A\|_{\mathcal{L}(\bba)}\leq \|A\|_{\mathcal{D}}\eq
	for all $A\in \mathcal{D}$.\\
	
	(ii) $\|AB\|_{\mathcal{D}}\leq \|A\|_{\mathcal{D}}\|B\|_{\mathcal{D}}$ for all  $A, B\in \mathcal{D}$.
	\end{defn}

For the sake of simplicity, we say that the subalgebra $\mathcal{D}$ is an {\em embedded subalgebra} if the norm on  $\mathcal{D}$ satisfies (i) and (ii).\\

If, in addition, $\mathcal{F}_{\mathcal{D}}:=\mathcal{F}\bigcap\mathcal{D}$ is dense in $\mathcal{D}$ with respect to the norm $\|\cdot\|_{\mathcal{D}}$, we say that $\mathcal{D}$ is {\em approximable}. \\

If $\mathcal{D}$ is an approximable algebra we can continuously extend the trace and determinant from finite rank operators to the algebra $\mathcal{D}$. 

\begin{thm}\label{th2eqr} Let $\mathcal{D}\subset\mathcal{L}(\bba)$ be an approximable embedded subalgebra. The following statements are equivalent:

(i) The function $\Det(I+F):\mathcal{F}_{\mathcal{D}}\rightarrow\ce$ admits a continuous extension in the $\mathcal{D}$-norm from $\mathcal{F}_{\mathcal{D}}$ to $\mathcal{D}$.\\

 (ii) The linear functional $\Tr(F) $ is bounded in the $\mathcal{D}$-norm on $\mathcal{F}_{\mathcal{D}}$.

	\end{thm}

The conditions above are satisfied in a good number of important examples. First we point out  how the conditions above are used to extend the trace and determinants from finite rank operators.\\

Let $\mathcal{D}\subset\mathcal{L}(\bba)$ be an approximable embedded subalgebra and assume  $\mathcal{D}$
 satisfies one of the equivalent conditions of Theorem \ref{th2eqr}. If $A\in \mathcal{D}$ we define the {\em trace} of $A$ and the {\em determinant } $I+A$ in the algebra $ \mathcal{D}$ by the equalities:
 \beq \Tr_{\mathcal{D}}(A):=\lim\limits_{n\rightarrow\infty}\Tr(F_n) \mbox{ and } \Det\limits_{\mathcal{D}}(I+A):=\lim\limits_{n\rightarrow\infty}\Det(I+F_n) ,  \eq
 where $\|A-F_n\|_{\mathcal{D}}\rightarrow 0$ with $F_n\in \mathcal{F}_{\mathcal{D}}$. As a consequence of the continuity of the extended trace and determinant, it can be shown that these definitions are independent of the choice of an approximative sequence $F_n$.\\
 
We can now introduce the embedded algebras that we are going to consider in this work. We first  recall some generalities on nuclear operators. 

\begin{defn} We will say that a Banach space $\bba$ satisfies the  {\em approximation property} if for every compact set $K$ in $\bba$ and for every $\epsilon >0$, there exists $F\in \mathcal{F}(\bba) $ such that
	\[\|x-Fx\|<\epsilon,\,\, \mbox{for every }x\in K.\]
\end{defn}

It is well known that the classical spaces $C(X)$ where $X$ is a compact topological space, $L^p(\mu)$ for $1\leq p\leq \infty$ and any measure satisfy the approximation property (cf. \cite{piet:book1}). P. Enflo showed a counterexample to the approximation property in Banach spaces (cf. \cite{ap:enflo}). A more natural counterexample has been found by A. Szankowsky who has proved that $\mathcal{L}(H)$ does not have the approximation property (cf. \cite{ap:sz}). \\

We can now define the concepts of nuclear operators on Banach spaces and the trace. Let $\mathcal{B}$ be a Banach space, a linear operator $T$ from $\bba$ into $\bba$ is called {\em nuclear} if there exist sequences
$(x_{n}^{'})\mbox{ in } \bba '$ and $(y_n) \mbox{ in } \bba$ in such a way  that $T$ admits a nuclear decomposition, which means that,
\begin{equation}\label{nucleardecompo}
   Tx= \sum\limits_n \left <x,x_{n}'\right>y_n \,\mbox{ and }\,
\sum\limits_n \|x_{n}'\|_{\bba '}\|y_n\|_{\bba} < \infty. 
\end{equation}
This definition agrees with the concept of trace class operator in
the setting of Hilbert spaces. The set of nuclear operators from $\bba$ into $\bba$ forms the ideal of nuclear operators $\mathcal{N}(\bba)$ endowed with the norm
\beq\nonumber N(T)=\inf\{\sum\limits_n \|x_{n}'\|_{\bba '}\|y_n\|_{\bba} : T=\sum\limits_n x_{n}'\otimes y_n \}.\eq
 Since $\mathcal{N}(\bba)$  is an ideal  in $\mathcal{L}(\bba),$ it is also a subalgebra. 
It is natural to attempt to define the trace of $T\in\mathcal{N}(\bba)$ by
$$\Tr (T)=\sum\limits_{n=1}^{\infty}x_{n}'(y_n),$$
where $T=\sum\limits_{n=1}^{\infty}x_{n}'\otimes y_n$ is a
representation of $T$. Grothendieck (cf. \cite{gro:me}) discovered that $\Tr(T)$ is well defined for all $T\in\mathcal{N}(\bba)$ if and only if $\bba$ has the {\em aproximation property} (cf. \cite{piet:book}, \cite{df:tensor}). 

In particular, if $\bba$ has the aproximation property, then the subalgebra $\mathcal{N}(\bba)$ of nuclear operators is an approximable embedded subalgebra in $\mathcal{L}(\bba)$. 
We also recall that if $\bba=\Hcal$ is a complex separable  Hilbert space, then the ideal of nuclear operators $\mathcal{N}(\bba)$ coincides with the trace class.\\

In the setting of $L^p(\Omega, \mu)$ spaces, the nuclear operators are characterized by kernels of the form 
\beq k(x,y)=\sum \limits_{n=1}^{\infty}g_{n}(x)\otimes h_n(y),\label{char57}\eq
where $g_{n}, h_n$ are sequences in $L^p, L^q$ respectively and such that $\sum\limits_n \|g_{n}\|_{L^p}\|h_n\|_{L^q} < \infty$. For the above characterization the reader can consult \cite{del:tracetop}, \cite{del:trace}. Given  such kernel, the trace can be expressed in the form 
\beq\label{tr1djd}\Tr(T)=\int\limits_{\Omega}\left(\sum \limits_{n=1}^{\infty}g_{n}(x)h_n(x)\right) d\mu (x),\eq
which holds almost everywhere on the domain. In other words, the trace can be obtained by integration on the diagonal. But a trace formula for an arbitrary kernel $\alpha(x,y)$ of $T$ is not directly deducible from the corresponding one for $k$ using \eqref{tr1djd}. The problem of traceability and of finding trace formulae has been extensively investigated 
 (cf. \cite{fm:nucl},  \cite{mo:eigen}, \cite{dr13a:nuclp}, \cite{dr13:schatten}, \cite{dr:suffkernel}). In particular, if one disposes of a continuous kernel $\alpha(x,y)$ of a nuclear operator over a suitable topological space one has 
\[\Tr(T)=\int\limits_{\Omega}\alpha(x,x) d\mu (x).\]

On the other hand as it is well known that to characterize the nuclearity in terms of a given integral kernel is in general not an easy task. Indeed, there are some classical counterexamples that we will recall later on. However this obstruction vanishes in the case for instance of the counting measure on the lattice $\zet^n$ due to the uniqueness of the kernel. Moreover, one has the following mild sufficient condition  for the nuclearity of discrete operators  in terms of their kernels (that later we will use  in Theorem \ref{detzn}). Below we consider $e_j(i)=\delta_{ji}$ the Kronecker's delta. 
\begin{prop}\label{discret} Let $1\leq p<\infty$ and let $K:\zet^n\times\zet^n\rightarrow\ce $ be a function satisfying
		\beq  \sum\limits_{j\in\zet^n}\left(\sum\limits_{m\in\zet^n}|K(j,m)|^{p}\right)^{\frac{1}{p}}<\infty .\label{sc}
		\eq
	Then the relation $K(i,j)=\langle Te_j,e_i\rangle, (i,j)\in\zet^n\times\zet^n,$ defines a nuclear operator $T:\ell^{p}({\zet^n})\rightarrow \ell^{p}({\zet^n}) $. Moreover its trace is given by
	\[\Tr(T)=\sum\limits_{n\in\zet^n} K(n,n). \]
\end{prop}
\begin{proof} We consider $q$ such that $\sdual.$ We are going to prove that \eqref{sc} makes  the decomposition \eqref{nuclearrepresentation} a nuclear decomposition of $T$ (in the sense of \eqref{nucleardecompo}). Define the functional $e_{n}'\in (\ell^{q}(\mathbb{Z}^n))'$ by $e_{n}'(f):=\langle f,e_n\rangle.$ Observe that $\Vert e_n'\Vert_{(\ell^{q}(\mathbb{Z}^n))'}=\Vert e_n \Vert_{\ell^{p}(\mathbb{Z}^n)}=1,$ and that
	\begin{equation}\label{nuclear}
	\sum\limits_{j\in\zet^n}\| Te_j\|_{\ell^{p}}
	\|e_j\|_{\ell^p}=\sum\limits_{j\in\zet^n}\| Te_j\|_{\ell^p}=\sum\limits_{j\in\zet^n}\left(\sum\limits_{m\in\zet^n}|K(m,j)|^{p}\right)^{\frac{1}{p}}<\infty.
	\end{equation}This justifies the following calculation for $f\in \ell^{p}({\zet^n})$:
	\begin{align*} 
	f&=\sum\limits_{j\in\zet^n}\langle f,e_j\rangle e_j,
	\end{align*}
	\begin{equation}\label{nuclearrepresentation}
	    Tf=\sum\limits_{j\in\zet^n}\langle f,e_j\rangle Te_j,
	\end{equation}
where the right hand side of \eqref{nuclearrepresentation} converges in the $\ell^p$-norm. Indeed, $T$ is bounded  on $\ell^{p}(\mathbb{Z}^n)$ as it proves  the following estimate 
\begin{equation}
    \Vert Tf\Vert_{\ell^p}\leq \left( \sum\limits_{j\in\mathbb{Z}^n}\| Te_j\|_{\ell^{p}}\right) \Vert f\Vert_{\ell^p}.
\end{equation}
Now, in view of \eqref{nuclear}, $T$ is a nuclear operator from $\ell^{p}({\zet^n})$ to
	$\ell^{p}({\zet^n})$, with kernel $K(i,j)=\langle Te_j,e_i\rangle$. Indeed, the kernel of $T$ is obtained by writing 
\begin{eqnarray*}
	Tf(i)&=&\sum\limits_{j\in\zet^n}\langle f,e_j\rangle \langle Te_j,e_i\rangle \\
	&=&\sum\limits_{j\in\zet^n}f(j)\langle Te_j,e_i\rangle ,
	\end{eqnarray*}
where we have used 	the identity $(Te_j)(i)=\langle Te_j,e_i\rangle.$
 	Therefore the kernel of $T$ is given by $K(i,j)=\langle Te_j,e_i\rangle$. Thus, we end the proof.
\end{proof}
The inequality \eqref{sc} in the above proposition is not a necessary condition for nuclearity. We distinguish two cases, $1<p<\infty $ or $p=1$. In the first case, we define a rank one operator with kernel $K(j,m)=g(j)h(m)$, where $g\in \ell^{p}({\zet})\setminus  \ell^{1}({\zet})$ and $h\in \ell^{q}({\zet})\setminus \{0\}$. Then the operator $T$ defined by $K$ is nuclear. However, 
\begin{align*}\sum\limits_{j\in\zet}(\sum\limits_{m\in\zet}|K(j,m)|^{p})^{\frac{1}{p}}&=\sum\limits_{j\in\zet}|g(j)|(\sum\limits_{m\in\zet}|h(m)|^{p})^{\frac{1}{p}}\\
&=\|g\|_{\ell^1}\|h\|_{\ell^{p}}\\
&=\infty.
\end{align*}
The other case is similar, if $p=1$ we choose $g\in \ell^{p}({\zet})\setminus \{0\}$ and $h\in \ell^{\infty}({\zet})\setminus  \ell^{p}({\zet})$, having in account that $p<\infty$, and $q=\infty$.\\

For Hilbert spaces we now recall the definition of the trace class. Let $\Hcal$ be a complex separable Hilbert space and let $T:\Hcal\rightarrow \Hcal$ be a linear compact operator.  We denote by  $s_k(T)$ the singular values of $T$, i.e, the eigenvalues of $|T|=\sqrt{T^*T}$. If 
$$
\sum_{k=1}^{\infty} s_k(T)<\infty ,
$$
then the linear operator $T:\Hcal\rightarrow \Hcal$ is said to be in the {\em trace class} $S_1(H)$. It can be shown that  $S_1(\Hcal)$ is a Banach space in which the norm $\|\cdot\|_{S_1}$ is given by 
$$
\|T\|_{S_1}= \sum_{k=1}^{\infty} s_k(T),\,T\in S_1,
$$
multiplicities counted.
Let $T:\Hcal\rightarrow \Hcal$ be an operator in $S_1(\Hcal)$ and let  $(\phi_k)_k$ be any orthonormal basis for $\Hcal$. Then, the series $\sum\limits_{k=1}^{\infty} \p{T\phi_k,\phi_k}$   is absolutely convergent and the sum is independent of the choice of the orthonormal basis $(\phi_k)_k$. Thus, we can define the trace $\Tr(T)$ of any linear operator
$T:\Hcal\rightarrow \Hcal$ in $S_1$ by 
$$
\Tr(T):=\sum_{k=1}^{\infty}\p{T\phi_k,\phi_k},
$$
where $\{\phi_k: k=1,2,\dots\}$ is any orthonormal basis for $\Hcal$. 

\section{Determinant formulae}\label{sec3}
In this section we present the main results of this work. 
We start by recalling a Plemelj-Smithies formula. For a compact manifold without boundary $M,$ $\Psi^{m}(M,\textnormal{loc})$ denotes the standard H\"ormander class of pseudo-differential operators of order $m\in \mathbb{R},$ (cf. \cite{Hormander1985III}).

\begin{thm}\label{thm1af}  Let $\mathcal{D}\subset\mathcal{L}(\bba)$ be an approximable embedded subalgebra. Suppose that the function $\Det(I+F)$ has a continuous extension from from $\mathcal{F}_{\mathcal{D}}$ to $\mathcal{D}$. Then for any $A\in\mathcal{D}$ fixed, the function $\Det\limits_{\mathcal{D}}(I+\lambda A)$ is an entire function and 
\[\Det_{\mathcal{D}}(I+\lambda A)=\exp\left(\sum\limits_{j=1}^{\infty}\frac{(-1)^{j+1}}{j}\Tr_{}(A^j)\lambda ^j\right),  \]
for $\lambda$ in a sufficiently small disc centered at $0$ in $\ce$. 
\end{thm}

We now consider different embedded subalgebras and the corresponding results.

\subsection{Poincar\'e determinant  on the torus}
We will present an analysis for the determinant of the toroidal quantisation in the Kohn-Nirenberg toroidal classes $S^m(\tn\times \zn)$ (cf. \cite{rt:book}).
We first note that if $A\in \ell^{1}(\zn\times\zn)$, we can write $A=(a_{jk})_{j,k\in \mathbb{Z}^n}$ with
\beq\sum\limits_{j,k\in \mathbb{Z}^n}|a_{jk}|<\infty.\label{inpoi}\eq 
The matrix $A$ can be viewed as the matrix with respect to the standard basis of a bounded operator (which we again denote by $A$) on $\bba=\ell^{p}(\zn)$ for all $1\leq p<\infty$. Indeed, in view of the discrete Schur test,  the operator
\begin{equation*}
    Af(k):=\sum_{j\in \mathbb{Z}^n} a_{jk}f(j),\,\,f\in c_{0}(\mathbb{Z}^n):=\{f:\mathbb{Z}^n\rightarrow\mathbb{C}:\,f\textnormal{  has compact support}\}
\end{equation*}extends to a bounded linear operator on $\ell^p(\mathbb{Z}^n),$ for all $1\leq p\leq \infty,$ and $$\Vert {A}\Vert_{\ell^p\rightarrow \ell^p}\leq \left(\sup_{k\in \mathbb{Z}^n}\Vert a_{jk}\Vert_{\ell^1(\mathbb{Z}^n)}\right)^{\frac{1}{p}}\left(\sup_{j\in \mathbb{Z}^n}\Vert a_{jk}\Vert_{\ell^1(\mathbb{Z}^n)}\right)^{1-\frac{1}{p}}. $$ 
The Poincar\'e algebra  $\Dcal_p$ consists of those bounded linear  operators $A$ on $\ell^p(\mathbb{Z}^n)$ with  matrices (with respect to the standard basis) in $\ell^{1}(\zn\times\zn)$. We endow $\Dcal_p$ with the norm
\[\|A\|_{\Dp}:=\sum\limits_{j,k\in \zn}|a_{jk}|.\,\,\,\,\,\,\,\,\,\,\,\,\,\,\,\,\,\,\,\,\,\,\,\,\,\,\,\,\,\,\,\,\]
The value $p$ indicates that the operator is considered acting on the space $\bba=\ell^{p}(\zn)$ despite the norm is independent of $p$.\\

For a fixed $p$ such that $1\leq p<\infty$, one can prove that $\Dp$ is indeed an approximable embedded subalgebra of $\mathcal{L}(\bba)$ with 
 $\bba=\ell^{p}(\zn)$.\\

We observe that
\beq |\Tr(A)|=\left|\sum\limits_{j\in \zn}a_{jj}\right|\leq \|A\|_{\Dp}.\,\,\,\,\,\,\,\,\,\,\,\,\eq
Thus, $\Dp$ satisfies the assumption (ii) of Theorem \ref{th2eqr}, and therefore the determinant $\Det(I+F)$ can be continuously extended to $\Dp$. We call this determinant the {\em Poincar\'e determinant.}\\
\begin{rem}
The condition $A\in \ell^1(\mathbb{Z}^n\times \mathbb{Z}^n),$ can be improved to the following one $A\in( \ell^\infty(\mathbb{Z}^n)\otimes \ell^1(\mathbb{Z}^n))\cap( \ell^\infty(\mathbb{Z}^n)\otimes \ell^1(\mathbb{Z}^n) ),$ in order to obtain a bounded operator on $\ell^p(\mathbb{Z}^n)$ (see Lemma 3.3.30 of \cite{rt:book}). However, we keep the  algebra condition on $A\in \ell^1(\mathbb{Z}^n\times \mathbb{Z}^n)$ because for $n=1,$ we recover the usual algebra $\mathcal{D}_p$ on $\ell^p(\mathbb{Z})$ introduced by Poincar\'e (cf. \cite{goh:trace}).
\end{rem}
We now consider operators on the torus $\tn$. In this special domain we can  take advantage of the Fourier transform and the duality with $\zn$ to relate the operators on the torus with the Poincar\'e determinant.\\

Let $\tn=\Rn/\zn$ be the $n$-dimensional torus. The collection $\{e^{2\pi ix\cdot k}\}_{k\in\zn}$ is an orthonormal basis of $L^2(\tn)$. We denote the Fourier transform on the torus $\tn$ by $\fou_{\tn}$ which is defined by
\[(\fou_{\tn}\varphi)(k)=\int\limits_{\tn}e^{-2\pi ix\cdot k}\varphi(x)dx,\,\,\varphi\in L^1(\tn).\]
The toroidal quantisation associated to a symbol $\sigma\in S^m(\tn\times\zn)$ is the densely defined  operator $T=\sigma(x,D)$ given by 
\[ \,\,\,\,  \,\,\,\,\,\,\,\,\,\,\,\,\,\,\,\,Tf(x)=\sum\limits_{k\in\zn}e^{2\pi ix\cdot k}\sigma(x,k)\widehat{f}(k) ,\,\,\, f\in C^\infty(\tn).\]
This quantisation has been extensively analysed in \cite{rt:book} for the general case of $\Tn$ and on compact Lie groups. For the toroidal H\"ormander class of order $m\in \mathbb{R},$ one has  $\Psi^m(\mathbb{T}^n,\textnormal{loc})=\{\sigma(x,D):\sigma\in S^m(\mathbb{T}^n\times \zet^n)\}$ (cf. \cite{rt:book}). 
 We observe that for $f\in C^{\infty}(\tn)$, and after a change of variable we have 
\begin{align*}
\,\,\,\,\,\,\,\,\,\,\,\,\,\,\,\,\,\,\,\,Tf(x)=&\sum\limits_{k\in\zn}e^{2\pi ix\cdot k}\sigma(x,k)\widehat{f}(k)\\
=&\sum\limits_{k\in\zn}\left(\sum\limits_{l\in\zn}\widehat{\sigma}(l,k)e^{2\pi ix\cdot l}\right)\widehat{f}(k)e^{2\pi ix\cdot k}\\
=&\sum\limits_{j\in\zn}\left(\sum\limits_{k\in\zn}\widehat{\sigma}(j-k,k)\widehat{f}(k)\right) e^{2\pi ix\cdot j}.
\end{align*}
From the last identity we can write $Tf$ as the Fourier series
\[Tf(x)=\sum\limits_{j\in\zn}C_{j} e^{2\pi ix\cdot j},\,x\in \tn,\,\,\,\,\,\,\,\,\,\,\,\,\,\,\,\,\,\,\,\,\,\,\,\,\,\,\,\]
where $C_j$ is the $j$-th Fourier coefficient of $Tf$.
We  now note that since  
\[C_j=\sum\limits_{k\in\zn}\widehat{\sigma}(j-k,k)\widehat{f}(k),\]
 we can write 
$
    C_j=\sum\limits_{k\in\zn}A_{jk}\widehat{f}(k) =\sum\limits_{k\in\zn}A_{jk}\phi(k),
$
where $A_{jk}$ is the matrix defined by 
\beq\label{amatf8} A_{jk}:=\widehat{\sigma}(j-k,k)\,\,\,\,\,\,\,\,\,\,\,\,\,\,\,\,\,\,\,\,\,\,\,\,\,\,\, \eq
and $\phi(k)=\widehat{f}(k).$\\

We now observe  that if the matrix  $A$ belongs to $\ell^1(\zn\times\zn)$, then $A$ defines an operator in $\Dcal_p$. In particular we consider $p=2$ to take advantage of the isometry of the Fourier transform. On the other hand the operator $T$ can be written in the form 
$T=\fou_{\tn}^{-1}A\fou_{\tn}.$
We observe that 
$\Gamma(A):=\fou_{\tn}^{-1}A\fou_{\tn}$
defines an algebra isomorphism 
\[\Gamma:\Dcal_2\longrightarrow\mathcal{L}(L^2(\To)).\]
To see the question on the approximabilitiy, we point out that the isomorphism $\Gamma$ preserves the class of finite rank operators $\mathcal{F}_{\Dcal_2}(\ell^2)=\mathcal{F}(\ell^2)$, that is   \[\Gamma(\mathcal{F}(\ell^2))=\mathcal{F}(L^2(\To)).\]
The above holds since given $H\in\mathcal{F}(L^2(\tn))$ we can write
$H=\fou_{\tn}^{-1}(\fou_{\tn}H\fou_{\tn}^{-1})\fou_{\tn}.$
Hence $\Gamma(\fou_{\tn}H\fou_{\tn}^{-1})=H.$ In $\Gamma(\Dcal_2)$ the norm considered is the carried one from $\Dcal_2$, i.e. 
\[\|T\|_{\Gamma}:=\|A\|_{\Dcal_2}.\]
Since the Poincar\'e determinant is defined on  $\Dcal_2$ as well as the trace, then the same can be done on the embedded subalgebra $\Gamma(\Dcal_2)$ of operators on $L^2(\tn)$. \\

Summarising the discussion above, we have:

\begin{thm}
	The Poincar\'e determinant on $\Dcal_2$ induces a determinant on $\Gamma(\Dcal_2)$ through the formula
	\beq\Det(I+T)=\Det(I+A).\eq
	\end{thm}
We will keep calling the induced determinant on  $\Gamma(\Dcal_2)$, the {\em Poincar\'e determinant } this time for operators on  $L^2(\tn)$, but to distinguish the domains we will denote it by $\Det_{\Gamma}$. We compute the Poincar\'e determinant for operators on the torus as follows.
\begin{thm} Let  $\sigma\in S^{\nu}(\tn\times\zn)$ with $\nu<-n.$ Then $T\in \Gamma(\mathcal{D}_2)$ and 
\begin{align}
   \textnormal{Det}(I+\lambda T)=\exp\left(\sum\limits_{m=1}^{\infty}\frac{(-1)^{m+1}}{m}\lambda ^m \sum_{  \begin{subarray}{l} (j_0,j_1,\cdots,j_{m}) \in\mathbb{Z}^{nm}\\ \,\,\,\,\,\,\,\,\,\,\,\,\,\,\,\,\,j_{0}=j_m  \end{subarray}}  \prod_{s=1}^{m}\widehat{\sigma}(j_{s-1}-j_{s},j_{s})  \right),
\end{align}for $\lambda$ in a sufficiently small disc centered at $0.$ Moreover, the index  in the condition $\nu<-n$ is sharp. 
\end{thm}\label{poincaretn}
\begin{proof} Let    $\sigma\in S^{\nu}(\mathbb{T}^n\times\mathbb{Z}^n).$ In view of Lemma 4.2.1 of \cite{rt:book}, for all $r\in \mathbb{R},$ there exists $C_r>0,$ such that
\begin{equation*}
    |\widehat{\sigma}(j,k)|\leq C_{r}(1+|j|)^{-r}(1+|k|)^{\nu}.
\end{equation*}If $\nu<-n,$ and $r>n,$ then
\begin{align*}
    \Vert T\Vert_{\Gamma(\mathcal{D}_2)}&=\sum_{j,k\in \mathbb{Z}^n}|A_{jk}|=\sum_{j,k\in \mathbb{Z}^n}|\widehat{\sigma}(j-k,k)|\\
     &\lesssim \sum_{j,k\in \mathbb{Z}^n}(1+|j-k|)^{-r}(1+|k|)^{\nu}<\infty.
\end{align*}Let us compute the determinant of $T.$ 
For $m\geq 2,$ observe that
\begin{align*}
    \textnormal{Tr}(A^m)&=\sum_{j_0\in \mathbb{Z}^n}[A^{m}]_{j_{0}j_{0}}\\
   &=\sum_{j_{0}\in \mathbb{Z}^n}\sum_{j_{1}\in \mathbb{Z}^n} A_{j_{0}j_{1}}[A^{m-1}]_{j_{1}j_{0}}\\
   &= \sum_{j_{0}\in \mathbb{Z}^n}\sum_{j_{1}\in \mathbb{Z}^n}\sum_{j_{2}\in \mathbb{Z}^n} A_{j_{0}j_{1}}A_{j_{1}j_{2}}[A^{m-2}]_{j_{2}j_{0}}\\
  &= \sum_{(j_0,j_1,\cdots,j_{m}) \in\mathbb{Z}^{nm}} A_{j_{0}j_{1}}A_{j_{1}j_{2}}\cdots A_{j_{m-1}j_{m}}\\
  &= \sum_{  \begin{subarray}{l} (j_0,j_1,\cdots,j_{m}) \in\mathbb{Z}^{nm}\\ \,\,\,\,\,\,\,\,\,\,\,\,\,\,\,\,\,j_{0}=j_m  \end{subarray}}  \prod_{s=1}^{m}A_{j_{s-1}j_{s}}\\
   &= \sum_{  \begin{subarray}{l} (j_0,j_1,\cdots,j_{m}) \in\mathbb{Z}^{nm}\\ \,\,\,\,\,\,\,\,\,\,\,\,\,\,\,\,\,j_{0}=j_m  \end{subarray}}  \prod_{s=1}^{m}\widehat{\sigma}(j_{s-1}-j_{s},j_{s}).
\end{align*} Now, we conclude the determinant formula by using Theorem \ref{thm1af}. To see the sharpness of the order condition $\nu<-n$ let us analyse the operator $T$ with symbol $\sigma_0(x,k):=(1+|k|)^{-n}\in S^{m'}(\mathbb{T}^n\times \mathbb{Z}^n)$ for all $m'\geq -n$. In this case, $\widehat{\sigma}_0(j-k,k)=(1+|k|)^{-n}\delta_{0,j-k}=(1+|k|)^{-n}\delta_{j,k}.$ Observe that 
\begin{align*}
    \Vert T\Vert_{\Gamma(\mathcal{D}_2)}&=\sum_{j,k\in \mathbb{Z}^n}|A_{jk}|=\sum_{k\in \mathbb{Z}^n}(1+|k|)^{-n}=\infty,
    \end{align*}which implies that $T\notin \Gamma(\mathcal{D}_2).$ Thus, we end the proof.
\end{proof}
\begin{rem} If  $\sigma\in S^{\nu}(\tn\times\zn)$ with $\nu<-n,$ then $T\in S_{1}(L^2(\tn))$ and one has
\begin{align}
   \textnormal{Det}(I+\lambda T)=\exp\left(\sum\limits_{m=1}^{\infty}\frac{(-1)^{m+1}}{m}\lambda ^m \sum_{  \begin{subarray}{l} (j_0,j_1,\cdots,j_{m}) \in\mathbb{Z}^{nm}\\ \,\,\,\,\,\,\,\,\,\,\,\,\,\,\,\,\,j_{0}=j_m  \end{subarray}}  \prod_{s=1}^{m}\widehat{\sigma}(j_{s-1}-j_{s},j_{s})  \right),
\end{align}for $\lambda$ in a sufficiently small disc centered at $0.$
\end{rem}
\begin{rem}\label{rem:poincare} On a separable Hilbert space $\mathcal{H},$ if a bounded linear operator $T:\mathcal{H}\rightarrow \mathcal{H} $ satisfies
\begin{equation*}
    \sum_{k,j\in J}|(Te_{k},e_{j})|<\infty,
\end{equation*} for some orthonormal basis $\{e_{j}\}_{j\in J}$ of $H,$ then  $T\in S_{1}(\mathcal{H})$ (cf. \cite[Page 172]{Sug90}). Consequently,  with the particular choice of the canonical orthonormal basis of $\ell^2(\mathbb{Z}^n),$ the condition $A\in \mathcal{D}_{2}$ implies that $A\in S_{1}(\ell^2(\mathbb{Z}^n)).$ Because  $\mathcal{D}_{2}\subset S_{1}(\ell^2(\mathbb{Z}^n)), $ we also have the  inclusion $\Gamma(\mathcal{D}_{2})\subset S_{1}(L^2(\mathbb{T}^n)) $ of the Poincar\'e algebra into the ideal of trace class operators on $L^2(\mathbb{T}^n).$ In view of Theorem \ref{poincaretn} and Remark \ref{rem:poincare}, we have
$$ \Psi^{\nu}(\mathbb{T}^n,\textnormal{loc})\subset \Gamma(\mathcal{D}_{2})\subset S_{1}(L^2(\mathbb{T}^n)),  $$
for all $\nu<-n.$
\end{rem}
\begin{rem} The inclusion $\mathcal{D}_{2}\subset S_{1}(\ell^2(\mathbb{Z}^n)) $ is strict. Indeed, consider the diagonalisable operator $A=(a_{jk})_{j,k\in \mathbb{Z}}$ in the canonical basis of $\ell^{2}(\mathbb{Z}),$ whose entries   $a_{jk}$ are given by $a_{1k}=1/k,$ $k\geq 1,$ $a_{jj}=1/j^2,$ $j\in \mathbb{Z}\setminus\{0\},$ and $a_{ij}=0$ in other case. The system of eigenvalues of $A$ is the set $\{1/j^2:j\in \mathbb{Z}\},$ which implies that $A\in S_{1}(\ell^2(\mathbb{Z}^n)),$ but $\Vert A\Vert_{\mathcal{D}_2}\geq \sum_{k= 1}^{\infty}1/k=\infty.$ This analysis also leads to the strict inclusion $\Gamma(\mathcal{D}_{2})\subset S_{1}(L^2(\mathbb{T}^n)). $

\end{rem}

\subsection{Determinant of nuclear operators on $\ell^p(\mathbb{Z})$} Having proved the trace formula for nuclear operators in Theorem \ref{discret}, in the next result we present the corresponding determinant formula.
\begin{thm}\label{detzn} Let $1\leq p<\infty$ and let $K:\zet\times\zet\rightarrow\ce $ be a function satisfying
	\beq  \sum\limits_{j\in\zet}\left(\sum\limits_{m\in\zet}|K(j,m)|^{p}\right)^{\frac{1}{p}}<\infty .\label{sc2'}\eq
	Then the relation $K(i,j)=\langle Te_j,e_i\rangle, (i,j)\in\zet\times\zet$ defines a nuclear operator $T:\ell^{p}({\zet})\rightarrow \ell^{p}({\zet}) $. Moreover, one has \begin{align}\label{sc2}
   \textnormal{Det}(I+\lambda T)=\exp\left(\sum\limits_{m=1}^{\infty}\frac{(-1)^{m+1}}{m}\lambda ^m \sum_{  \begin{subarray}{l} (j_0,j_1,\cdots,j_{m}) \in\mathbb{Z}^m\\ \,\,\,\,\,\,\,\,\,\,\,\,\,\,\,\,\,j_{0}=j_m  \end{subarray}}  \prod_{s=1}^{m}K(j_{s-1},j_{s})  \right),
\end{align}for $\lambda$ in a sufficiently small disc centered at $0.$
\end{thm}
\begin{proof}
Straightforward computation shows that for $m\geq 2,$ the kernel of $T^m$ is given by
\begin{equation*}
    K_{T^m}(j_{0},j_{m})=\sum_{(j_1,\cdots,j_{m-1})}K(j_0,j_1)K(j_1,j_2)\cdots K(j_m,j_{m-1}).
\end{equation*} So, summing $K_{T^m}(j_{0},j_{m})$ over $j_0=j_1\in \mathbb{Z}$ we obtain the trace of $T^m,$ (this, in view of Proposition \ref{discret} and \eqref{sc2} follows from Theorem \ref{thm1af}.
\end{proof}
\subsection{Determinant of invariant operators on Hilbert spaces}

We now recall the notion of invariant operators introduced in \cite{dr14a:fsymbsch} and which is based on the following theorem:
\begin{thm}\label{THM:inv-rem}
Let $\Hcal$ be a complex Hilbert space and let $\Hcal^{\infty}\subset \Hcal$ be a dense
linear subspace of $\Hcal$. Let $\{d_{j}\}_{j\in\N_{0}}\subset\N$ and let
$\{e_{j}^{k}\}_{j\in\N_{0}, 1\leq k\leq d_{j}}$ be an
orthonormal basis of $\Hcal$ such that
$e_{j}^{k}\in \Hcal^{\infty}$ for all $j$ and $k$. Let $H_{j}:={\rm span} \{e_{j}^{k}\}_{k=1}^{d_{j}}$,
and let $P_{j}:\Hcal\to H_{j}$ be the orthogonal projection.
For $f\in\Hcal$, we denote $\widehat{f}(j,k):=(f,e_{j}^{k})_{\Hcal}$ and let
$\widehat{f}(j)\in \ce^{d_{j}}$ denote the column of $\widehat{f}(j,k)$, $1\leq k\leq d_{j}.$
Let $T:\Hcal^{\infty}\to \Hcal$ be a linear operator.
Then the following
conditions are equivalent:
\begin{itemize}
\item[(A)] For each $j\in\ene_0$, we have $T(H_j)\subset H_j$. 
\item[(B)] For each $\ell\in\ene_0$ there exists a matrix 
$\sigma_{T}(\ell)\in\ce^{d_{\ell}\times d_{\ell}}$ such that for all $e_j^k$ 
$$
\widehat{Te_j^k}(\ell,m)=\sigma_{T}(\ell)_{mk}\delta_{j\ell}.
$$
\item[(C)]  For each $\ell\in\ene_0 $ there exists a matrix 
$\sigma_{T}(\ell)\in\ce^{d_{\ell}\times d_{\ell}}$ such that
 \[\widehat{Tf}(\ell)=\sigma_{T}(\ell)\widehat{f}(\ell)\]
 for all $f\in\Hcal^{\infty}.$
\end{itemize}

The matrices $\sigma_{T}(\ell)$ in {\rm (B)} and {\rm (C)} coincide.

The equivalent properties {\rm (A)--(C)} follow from the condition 
\begin{itemize}
\item[(D)] For each $j\in\ene_0$, we have
$TP_j=P_jT$ on $\Hcal^{\infty}$.
\end{itemize}
If, in addition, $T$ extends to a bounded operator
$T\in{\mathscr L}(\Hcal)$ then {\rm (D)} is equivalent to {\rm (A)--(C)}.
\end{thm} 

Under the assumptions of Theorem \ref{THM:inv-rem}, we have the direct sum 
decomposition
\begin{equation}\label{EQ:sum}
\Hcal = \bigoplus_{j=0}^{\infty} H_{j},\quad H_{j}={\rm span} \{e_{j}^{k}\}_{k=1}^{d_{j}},
\end{equation}
and we have $d_{j}=\dim H_{j}.$
The two applications that we will consider will be with $\Hcal=L^{2}(M)$ for a
compact manifold $M$ with $H_{j}$ being the eigenspaces of an elliptic 
pseudo-differential operator $E$, or with $\Hcal=L^{2}(G)$ for a compact Lie group
$G$ with $H_{j}=\textrm{span}\{\xi_{km}\}_{1\leq k,m\leq d_{\xi}}$ for a
unitary irreducible representation $\xi\in[\xi_{j}]\in\widehat{G}$. In this case we will write $\sigma_{T}:=\sigma_{T,E}.$  The difference
is that in the first case we will have that the eigenvalues of $E$ corresponding
to $H_{j}$'s are all distinct, while in the second case the eigenvalues of the Laplacian
on $G$ for which $H_{j}$'s are the eigenspaces, may coincide.

In view of properties (A) and (C), respectively, an operator $T$ satisfying any of
the equivalent properties (A)--(C) in
Theorem \ref{THM:inv-rem}, will be called an {\em invariant operator}, or
a {\em Fourier multiplier relative to the decomposition
$\{H_{j}\}_{j\in\N_{0}}$} in \eqref{EQ:sum}.
If the collection $\{H_{j}\}_{j\in\N_{0}}$
is fixed once and for all, we can just say that $T$ is {\em invariant}
or a {\em Fourier multiplier}.

The family of matrices $\sigma$ will be
called the {\em matrix symbol of $T$ relative to the partition $\{H_{j}\}$ and to the
basis $\{e_{j}^{k}\}$.}
It is an element of the space $\Sigma$ defined by
\begin{equation}\label{EQ:Sigma1}
\Sigma=\{\sigma:\N_{0}\ni\ell\mapsto\sigma(\ell)\in \ce^{d_{\ell}\times d_{\ell}}\}.
\end{equation}
In this section we will investigate the concept of determinant on embedded in the sense of \cite{goh:trace}   
for the notion of invariant operators introduced in \cite{dr14a:fsymbsch}. We recall below the Plemelj-Smithies formula:\\

\begin{thm}[Plemelj-Smithies formula]\label{thm0a}
If $T\in S_1(\Hcal)$ one has
\[\Det(I+\lambda T)=\exp\left(\sum\limits_{m=1}^{\infty}\frac{(-1)^{m+1}}{m}\Tr (T^m)\lambda ^m\right),  \]
for $\lambda\in \mathbb{C}$ with  $|\lambda|$ small enough. \\
\end{thm}
As an application of the notion of invariant operators given by Theorem \ref{THM:inv-rem} we have:   
\begin{thm}\label{thm1a}  If $T\in S_1(\Hcal)$ is a Fourier multiplier, 
one has
\[\Det(I+\lambda T)=\exp\left(\sum\limits_{m=1}^{\infty}\frac{(-1)^{m+1}}{m} \lambda ^m\sum\limits_{\ell=0}^{\infty} \Tr(\sigma_{T}(\ell)^m)  \right),  \]
for $\lambda\in \mathbb{C}$ with  $|\lambda|$ small enough. 
\end{thm}
\begin{proof}
We recall (cf. \cite{dr14a:fsymbsch}, Theorem 2.5) that if $T\in S_1(\Hcal)$ and is invariant then 
\begin{equation}\label{EQ:Sch2}
\Tr(T)=\sum\limits_{\ell=0}^{\infty}\Tr(\sigma_{T}(\ell)).
\end{equation}
Now, $T^m$ is also invariant and $T^m\in S_1(\Hcal)$ since $S_1(\Hcal)$ is an ideal. On the other hand
\[\sigma_{T^m}(\ell)=(\sigma_{T}(\ell))^m.\]
Therefore
\[\Tr(T^m)=\sum\limits_{\ell=0}^{\infty}\Tr(\sigma_{T}(\ell)^m).\]
The proof can be concluded by an application of the Plemelj-Smithies formula.
\end{proof}

\subsection{Plemelj-Smithies formulas on $L^2(M)$}\label{sec4}
In particular we will consider $\Hcal=L^{2}(M)$ for a
compact manifold without boundary $M$ with $H_{j}$ being the eigenspaces of an elliptic 
pseudo-differential operator $E$ of order $\nu$ on $M$. We denote by $\Psi^{\nu}(M)$ the H\"ormander class of pseudo-differential 
 operators of order $\nu\in\er$,
 i.e. operators which, in every coordinate chart, are operators in H\"ormander classes 
 on $\Rn$ with symbols
 in $S^\nu_{1,0}$, see e.g. \cite{shubin:r} or \cite{rt:book}.
 In this paper we will be using the class  $\Psi^{\nu}_{cl}(M)$ of classical operators, i.e. operators
 with symbols having (in all local coordinates) an asymptotic expansion of the symbol in
 positively homogeneous components (see e.g. \cite{Duis:BK-FIO-2011}).
  This means that for any local chart $U$, the operator $A\in \Psi^{\nu}_{cl}(M),$ has the form
\[Au(x)=\int\limits_{T^{*}_xU} e^{ix\xi}\sigma^A(x,\xi)\widehat{u}(\xi) d\xi\]
where  $\sigma^A(x,\xi)$  is a smooth function on $T^{*} U\cong U\times\mathbb{R}^n,$ $T^{*} _{x}U=\mathbb{R}^n$,  admitting an asymptotic expansion 
\begin{equation}\label{Eq:asym-homg}
\sigma^A(x,\xi)\sim \sum_{j=0}^{\infty}\sigma_{m-j}(x,\xi)
\end{equation}
where $\sigma_{m-j}(x,\xi)$ are homogeneous functions in $\xi\neq 0,$ of degree $m-j$ for $\xi$ far  from zero. The set of classical pseudo-differential operators of order $m$ is denoted by $\Psi^m_{cl}(M)$.

 Furthermore, we denote by $\Psi_{+}^{\nu}(M)$ the class of positive definite operators in 
 $\Psi^{\nu}_{cl}(M)$,
 and by $\Psi_{e}^{\nu}(M)$ the class of elliptic operators in $\Psi^{\nu}_{cl}(M)$. Finally, 
 $$\Psi_{+e}^{\nu}(M):=\Psi_{+}^{\nu}(M)\cap \Psi_{e}^{\nu}(M)$$ 
 will denote the  class of classical positive elliptic 
 pseudo-differential operators of order $\nu$.
 We note that complex powers of such operators are well-defined, see e.g.
 Seeley \cite{Seeley:complex-powers-1967}. 

\begin{cor} If $M$ be a closed manifold of dimension $n$, and $E\in\Psi_{+e}^{\nu}(M)$. If $A=(I+E)^{-\frac{\alpha}{\nu}}$ with $\alpha>n$ then
$A\in S_1(L^2(M))$ and 
\[\Det(I+\lambda A)=\exp\left(\sum\limits_{m=1}^{\infty}\frac{(-1)^{m+1}}{m} \lambda ^m \sum\limits_{j=0}^{\infty}d_{j}(1+\lambda_j)^{-\frac{\alpha m}{\nu}}\right),  \]
for $\lambda\in \mathbb{C}$ with  $|\lambda|$ small enough and where  $\lambda_j$ denotes the eigenvalues of $E$ with multiplicities $d_j$.
\end{cor}
\begin{proof} The fact that $A\in S_1(L^2(M))$ follows from (cf. \cite{dr:suffkernel}, Proposition 3.3). It is clear that $A$ is invariant relative to $E$. Now, the eigenvalues of $(I+E)^{-\frac{\alpha m}{\nu}}$ are $(1+\lambda_j)^{-\frac{\alpha m}{\nu}}$ with the same multiplicities $d_j$ of $E$. The result now follows from Theorem \ref{thm1a}.
\end{proof}

\subsection{Determinants on homogeneous vector bundles in $L^2(E)$}\label{sec6}
In this section, let us denote by $E$ the total space of a vector bundle $p:E\rightarrow M$ over a manifold. For a fixed vector space $V,$ $\{e_{i,V}\}_{1\leq i\leq \dim(V) }$ denotes an orthonormal  basis of $V.$

Let $G$ be a compact Lie group, $dx$ its Haar measure and $\widehat{G}$ its unitary dual. In this section we study the determinant of $G$-invariant operators on a homogeneous vector bundle over a homogeneous manifold $M=G/K$ (where $G$ is a  Lie group and $K$ is a closed subgroup of $G$). 

The spectral analysis of $G$-invariant operators on a homogeneous vector bundle $p:E\rightarrow M$ (we will introduce it later) in $L^2(E)$  can be reduced to the spectral  analysis of  left-invariant operators on $L^2(G,E_0),$ (with $E_0=p^{-1}(K)$). We will explain this connection in detail. If $K_{A}\in \mathscr{D}'(G,\textnormal{End}(E_0)),$ is the convolution kernel of a left-invariant operator $A$ on $L^2(G,E_0),$ then
\begin{equation*}
    Af(x)=\int\limits_{G}K_{A}(y^{-1}x)f(y)dy,\,\,\,f\in L^2(G,E_0).\,\,\,\,\,\,\,\,\,\,
\end{equation*}
Let $B_{E_0}=\{e_{i,E_0}\}_{i=1}^{d_\tau},$ $d_\tau=\dim(E_0),$ be an orthonormal basis of $E_0.$ For every $f\in C^{\infty}(G,E_0),$ we can write:
$
    f(x)=\sum_{i=1}^{d_\tau}f_{i}(x)\,e_{i,E_0},$ with $f_{i}(x):=(f(x),e_{i,E_0})_{E_0}.$ The Peter-Weyl theorem implies
\begin{equation}
    f(x)=\sum_{i=1}^{d_\tau}\sum_{[\xi]\in \widehat{G}}d_{\xi}\textnormal{Tr}[\xi(x)\widehat{f}_{i}(\xi)   ]e_{i,E_0}.\,\,\,\,\,\,\,\,\,\,\,\,\,\,\,\,\,\,\,\,\,\,\,\,\,\,\,\,\,\,
\end{equation}
Here $\widehat{f}_{i}(\xi)$ denotes the Fourier transform of $f_{i}$ at $\xi,$ which is defined by
$$ \widehat{f}_{i}(\xi)\equiv (\mathcal{F}_{G}f_{i})(\xi):=\int\limits_{G}f_i(x)\xi(x)^{*}dx,\,\,[\xi]\in \widehat{G}.\,\,\,\,\,\,  $$
 The quantisation process in \cite{rt:book}  allows us to construct a matrix-symbol $\sigma_A:\{1,2,\cdots, d_\tau\}^2\times \widehat{G}\rightarrow \cup_{[\xi]\in \widehat{G}}\mathbb{C}^{d_\xi\times d_\xi},$ such that
\begin{equation}\label{vector-valuedq}
    Af(x)=\sum_{i,r=1}^{d_\tau}\sum_{[\xi]\in \widehat{G}}d_{\xi}\textnormal{Tr}[\xi(x)\sigma_A(i,r,\xi)\widehat{f}_{i}(\xi)   ]e_{r,E_0}.\,\,\,\,\,\,\,\,\,\,\,
\end{equation}
Indeed,
$$Af(x)
   =\sum_{i=1}^{d_\tau}\int\limits_{G}K_{A}(y^{-1}x)f_i(y)dy=\sum_{i=1}^{d_\tau}\sum_{[\xi]\in \widehat{G}}\int\limits_{G}d_{\xi}\textnormal{Tr}[\xi(y)\widehat{f}_{i}(\xi)   ]K_{A}(y^{-1}x)e_{i,E_0}dy$$
    $$=\sum_{i=1}^{d_\tau}\sum_{r=1}^{d_\tau}\sum_{[\xi]\in \widehat{G}}\int\limits_{G}d_{\xi}\textnormal{Tr}[\xi(y)\widehat{f}_{i}(\xi)   ](K_{A}(y^{-1}x)e_{i,E_0},e_{r,E_0})_{E_0}e_{r,E_0}dy.
$$If $A_{ir}$ is the operator with right-convolution kernel $K_{A}^{i,r}(z):=(K_{A}(z)e_{i,E_0},e_{r,E_0})_{E_0}\in \mathscr{D}'(G),$ then 
 \begin{align*}
      Af(x)=\sum_{i=1}^{d_\tau}\sum_{r=1}^{d_\tau}\int\limits_{G}K_{A}^{i,r}(y^{-1}x)f_{i}(y)e_{r,E_0}dy=\sum_{i=1}^{d_\tau}\sum_{r=1}^{d_\tau}A_{ir}f_{i}e_{r,E_0}.
 \end{align*} Denoting  $$\sigma^{i,r}_{A}:=\sigma_{A}(i,r,\cdot,\cdot)\equiv (\mathcal{F}_{G}K_{A}^{i,r}):\widehat{G}\rightarrow \cup_{[\xi]\in \widehat{G}}\mathbb{C}^{d_\xi\times d_\xi},$$ the matrix-valued symbol of $A_{ir}$ (see  \cite{rt:book} for details), we have 
\begin{equation*}
    A_{ir}f_i(x)=\sum_{[\xi]\in \widehat{G}}d_{\xi}\textnormal{Tr}[\xi(x)\sigma_A(i,r,\xi)\widehat{f}_{i}(\xi)   ],
\end{equation*}from which we deduce \eqref{vector-valuedq}. 
In the next lemma we study the composition of vector valued left-invariant operators.
\begin{lem}
If ${A}$ and ${B}$ are left-invariant operators and bounded on $L^2(G,E_0),$ then $AB$ is also a left-invariant operator and bounded on $L^2(G,E_0)$ and its symbol is given by 
\begin{equation}\label{compositionAandB}
    \sigma_{BA}(i, s,\xi) =\sum_{r=1}^{d_\tau}\sigma_{B}(r, s, \xi) \sigma_{A}(i, r, \xi),\,\,1\leq i,s\leq d_\tau.
\end{equation}
\end{lem}
\begin{proof}
Observe that 
 \begin{align*}
    (B\circ A) f(x):= B[Af](x)&= \sum_{r=1}^{d_\tau} \sum_{s=1}^{d_\tau} B_{rs}(Af)_r(x) e_{s, E_0}. 
 \end{align*}
 Now, note that 
 \begin{align*}
     (Af)_r(x) := \langle Af(x), e_{r, F_0} \rangle &= \left\langle \sum_{i=1}^{d_\tau} \sum_{r'=1}^{d_\tau} A_{ir'} f_i(x) e_{r', F_0}, e_{r, F_0} \right\rangle \\&= \sum_{i=1}^{d_\tau}  A_{ir} f_i(x).
 \end{align*}
 Therefore, we obtain
 \begin{align*}
     B[Af](x)&= \sum_{r=1}^{d_\tau} \sum_{s=1}^{d_\tau} B_{rs} \left[\sum_{i=1}^{d_\tau}  A_{ir} f_i(x) \right] e_{s, E_0} \\&= \sum_{r=1}^{d_\tau} \sum_{s=1}^{d_\tau} \sum_{i=1}^{d_\tau} B_{rs} A_{ir} f_i(x)  e_{s, E_0} \\&=\sum_{r=1}^{d_\tau} \sum_{s=1}^{d_\tau} \sum_{i=1}^{d_\tau} \sum_{[\xi] \in \widehat{G}} d_\xi \textnormal{Tr}[\xi(x) \sigma_{B_{rs} A_{ir}}( \xi) \widehat{f}_i(\xi)]   e_{s, E_0} \\&=  \sum_{s=1}^{d_\tau} \sum_{i=1}^{d_\tau} \sum_{[\xi] \in \widehat{G}} d_\xi \textnormal{Tr}[\xi(x) \sum_{r=1}^{d_\tau} \sigma_{B_{rs} A_{ir}}( \xi) \widehat{f}_i(\xi)]   e_{s, E_0}.
 \end{align*}
 By the uniqueness of the symbol, we deduce  that 
 $$\sigma_{BA}(i,s, \xi)= \sum_{r=1}^{d_\tau} \sigma_{B_{rs} A_{ir}}( \xi).$$ By observing that $\sigma_{B_{rs} A_{ir}}( \xi)=\sigma_{B}(r, s, \xi) \sigma_{A}(i, r, \xi),$ for all $1\leq i,s\leq d_\tau,$ we finish the proof.
\end{proof}

By following \cite{Bott65},  we say that a vector bundle $p:E\rightarrow M$ is homogeneous, if the following two conditions are satisfied,
\begin{itemize}
    \item[(i)] $G$ acts from the left on $E$ and for every $g\in G,$ $g\cdot E_{x}=E_{g\cdot x},$ $x\in M.$ 
    \item[(ii)] For every  $g \in G,$ the induced map from $E_{g}$ into $E_{g\cdot x},$ is linear. 
\end{itemize}
As usual, a section $s$ on $E$ is defined by the identity $p\circ s=\textnormal{id}_{M}.$  The set of smooth sections on $E$ will be denoted by $\Gamma(E).$ $L^2(E)$ denotes the set of square-integrable sections with respect to a normalised $G$-invariant volume form. 

We record that a continuous linear  operator $\tilde{A}:\Gamma(E)\rightarrow \Gamma(E)$ is homogeneous, if it is  invariant under the action of $G$ on the space of section $\Gamma(E)$, that is,
\begin{equation}\label{homogenesousdefinition}
    \tilde{A}(g\cdot s)=g\cdot (\tilde{A}s),\,\,s\in \Gamma(E).
\end{equation} 
In that follows, we will use a standard construction (see \cite{Wallach1973}) that allows to identify the set of smooth sections on a homogeneous vector bundle  $p:E\rightarrow M$ with the subclass of vector-valued smooth functions on $E_{0}=p^{-1}(K),$ defined by
\begin{equation}
    C^\infty(G,E_0)^\tau=\{f\in C^\infty(G,E_0)\,\, |\forall g\in G, \,k\in K,\,\,f(gk)=\tau(k)^{-1}f(g)  \},
\end{equation}for some representation $\tau$ of $K.$ Under this identification, we will be able to study the determinant of an (homogenoeus) invariant operator $\tilde{A},$ by using the notion of matrix-valued symbols developed in \cite{rt:book}.  
\begin{ex}[{\bf Bott's construction}]\label{BottConstruction}  Let us consider the vector bundle $p:E\rightarrow M=G/K,$ with $E=G\times_{\tau} E_{0},$ that is, the semi-direct product between $G$ and $E_{0}$ with respect to $\tau: K\rightarrow \textnormal{GL}(E_0),$ where  $E_0$ is the representation space of $K$ associated with $\tau.$ In view of the compactness of $K,$ we can assume that $E_0=\mathbb{C}^{d_\tau}$    is the complex vector space of finite dimension $d_\tau$. We recall that  the semi-direct product $G\times_{\tau}E_0$ is the set of all cosets $(g,v)\cdot K=:[g,v],$ $g\in G,$ $v\in E_{0},$ defined by the right action of $K$ on $G\times E_{0}, $ $(g,v)\cdot k=(gk,\tau(k)^{-1}v).$ Consequently,  $(g,v)\cdot K=\{(g,v)\cdot k:k\in K\}.$ The projection $p:E\rightarrow M,$ is given by $p((g,v)\cdot K)=gK,$ and $p:E\rightarrow M=G/K,$  has a natural structure of a homogeneous vector bundle. 
\end{ex}
Now, let us record that every homogeneous vector bundle can be constructed (up to isomorphism) in the way described in the previous example  (see \cite{Wallach1973}).
\begin{thm}Let $p:E\rightarrow M=G/K$ be a homogeneous vector bundle. Let $E_0=p^{-1}(K)$ be the fiber at the identity. Then, there exists a representation $\tau:K\rightarrow \textnormal{GL}(E_0)$ of $K,$ such that $E\cong G\times_{\tau}E_0.$ Moreover,  the mapping $\varkappa_\tau:\Gamma(E)\rightarrow C^\infty(G,E_{0})^{\tau}, $ given by $$\varkappa_\tau(s)(g):=g^{-1}\cdot s(gK),$$ extends to a unitary mapping $\varkappa_\tau:L^2(E)\rightarrow L^2(G,E_{0})^{\tau}.$
\end{thm}
From here, let us make the identification $E\cong G\times_{\tau} E_0.$ Observe that  a bounded linear operator  $\tilde{A}:L^2(E)\rightarrow L^2(E)$  is unitarily equivalent to the bounded linear operator on $L^2(G,E_0)^{\tau}$ defined by 
\begin{equation}\label{unitarilyequivaklent}
    A:=\varkappa_\tau\circ \tilde{A}\circ \varkappa_\tau^{-1}.
\end{equation}  If $\tilde{A}$ is $G$-invariant  then ${A}$ is left-invariant. Moreover, the stability of the spectrum under unitary transformations implies that $\tilde{A}\in S_{1}(L^2(E),$ if and only if, ${A}\in S_{1}(L^2(G,E_0)^\tau),$ and the analysis in \cite{dr14a:fsymbsch} implies that 
\begin{equation}\label{TrA}
   \Tr({A})=\sum_{[\xi]\in \widehat{G}, \,1\leq i\leq d_\tau}d_\xi\Tr(\sigma_A(i,i,\xi))= \Tr(\tilde{A}).
\end{equation}

For our further analysis we need to compute the symbol of an integer power $A^{m}$ of a left-invariant operator $A.$ It will be done in the following lemma.
\begin{lem}\label{symbolAm} Let $m\in \mathbb{N}$ with $m\geq 2.$ Then, the symbol $\sigma_{A^m}$ of $A^{m}$ is given by
\begin{equation}\label{compositionAmtimes}
    \sigma_{A^m}(r_m,r_0,\xi)=\sum_{r_{1},r_2,\cdots,r_{m-1}=1}^{d_\tau}\prod_{j=1}^{m}\sigma_{A}(r_{j},r_{j-1},\xi),\,\,\,1\leq r_m,r_0\leq d_\tau.
\end{equation}
    
\end{lem}\begin{proof}For the proof we will use induction on $m.$ For $m=2,$ from \eqref{compositionAandB} we have
\begin{equation*}
   \sigma_{A^2}(r_2, r_0,\xi) =\sum_{r_1=1}^{d_\tau}\sigma_{A}(r_1, r_0, \xi) \sigma_{A}(r_2, r_1, \xi),\,\,1\leq r_2,r_0\leq d_\tau.  
\end{equation*}Now, let us fix  $m\in \mathbb{N}$ and let us assume that \eqref{compositionAmtimes} holds. From \eqref{compositionAandB} one has for $1\leq r_{m+1},r_0\leq d_\tau$ that,
\begin{align*}
    \sigma_{A^{m+1}}(r_{m+1},r_0,\xi)=\sum_{r_1=1}^{d_\tau}\sigma_{A}(r_1, r_0, \xi) \sigma_{A^m}(r_{m+1}, r_1, \xi).
\end{align*}In view of \eqref{compositionAandB} we have
\begin{equation*}
    \sigma_{A^m}(r_{m+1},r_{1},\xi)=\sum_{r_{2},r_3,\cdots,r_{m}=1}^{d_\tau}\prod_{j=1}^{m}\sigma_{A}(r_{j+1},r_{j},\xi),
\end{equation*}and consequently,
\begin{align*}
    \sigma_{A^{m+1}}(r_{m+1},r_0,\xi)&=\sum_{r_1=1}^{d_\tau}\sigma_{A}(r_1, r_0, \xi)\sum_{r_{2},r_3,\cdots,r_{m}=1}^{d_\tau}\prod_{j=1}^{m}\sigma_{A}(r_{j+1},r_{j},\xi)\\
    &=\sum_{r_1,r_{2},r_3,\cdots,r_{m}=1}^{d_\tau}\prod_{j=0}^{m}\sigma_{A}(r_{j+1},r_{j},\xi)\\
    &=\sum_{r_1,r_{2},r_3,\cdots,r_{m}=1}^{d_\tau}\prod_{j=1}^{m+1}\sigma_{A}(r_{j},r_{j-1},\xi).
\end{align*}Thus, we end the proof.
\end{proof} Now, we prove our determinant formula for $G$-invariant operators on homogeneous vector bundles.
\begin{thm}Let $\tilde{A}$ be a homogeneous and bounded operator on $L^2(E).$ Then, $\tilde{A}\in S_{1}(L^2(E))$ if and only if ${A}\in S_{1}(L^2(G,E_0)^\tau),$ $\Det(I+\lambda \tilde{A})=\Det(I+\lambda {A}),$  and 
 one has
\begin{equation}\label{traceofhomogebeous}
    \Det(I+\lambda \tilde{A})=\exp\left(\sum\limits_{m=1}^{\infty}\frac{(-1)^{m+1}}{m}\lambda^m\sum_{  \begin{subarray}{l} 1\leq r_0=r_m\leq d_\tau,\,[\xi]\in \widehat{G}\\1\leq r_{1},r_{2},\cdots, r_{m-1}\leq d_\tau \end{subarray}}d_\xi\Tr\left(\prod_{j=1}^{m}\sigma_{A}(r_{j},r_{j-1},\xi)\right)  \right), 
\end{equation}
for $\lambda\in \mathbb{C}$ with  $|\lambda|$ small enough. 
\end{thm}
\begin{proof} Observe that $\tilde{A}$ and $A$ have the same system of eigenvalues because of the stability of the spectrum  under unitary transformations.  This proves the first part of the theorem. By following the argument in the proof of Theorem \ref{thm1a}, for the proof of \eqref{traceofhomogebeous} we just need to compute $\Tr(A^m),$ for $m\in \mathbb{N}.$ In view of \eqref{TrA} we have,
\begin{equation}
   \Tr({A}^m)=\sum_{[\xi]\in \widehat{G}, \,1\leq r_{m}\leq d_\tau}d_\xi\Tr(\sigma_{A^m}(r_{m},r_m,\xi)).
\end{equation}But, by using Lemma  \ref{symbolAm} and the linearity of the trace we obtain
$$ \Tr(\sigma_{A^m}(r_{m},r_m,\xi))=\sum_{r_{1},r_2,\cdots,r_{m-1}=1}^{d_\tau}\Tr\left(\prod_{j=1}^{m}\sigma_{A}(r_{j},r_{j-1},\xi)\right) ,  $$ where in the right hand side,   $r_0=r_m,$
completing the proof.
\end{proof}

\end{document}